\pdfoutput=1
\documentclass[11pt]{amsart}

\usepackage{epsfig,overpic,comment}
\usepackage[usenames,dvipsnames,svgnames,table]{xcolor}
\usepackage[pagebackref,linktocpage=true,colorlinks=true,linkcolor=Blue,citecolor=BrickRed,urlcolor=RoyalBlue]{hyperref}
\usepackage[alphabetic,msc-links,abbrev]{amsrefs} 
\usepackage[colorinlistoftodos,prependcaption,textsize=tiny, textwidth=0.75in]{todonotes}
\usepackage[a4paper, left=1.2in,right=1.2in,top=1.4in,bottom=1.7in]{geometry} 
\usepackage{esint}

\newtheorem{thm}{Theorem}[section]
\newtheorem{lem}[thm]{Lemma}
\newtheorem{cor}[thm]{Corollary}
\newtheorem{prop}[thm]{Proposition}

\newtheorem{claim}[thm]{Claim}
\newtheorem{remark}[thm]{Remark}

\newtheorem{definition}[thm]{Definition}

\newcommand{\R}{\mathbb{R}}

\newcommand{\ol}{\overline}
\newcommand{\K}{\mathcal{K}}   
\newcommand{\M}{\mathcal M}   

\newcommand\e{\varepsilon}
\newcommand\na{\nabla}
\newcommand\p{\partial}

\renewcommand{\sf}{\mathbb{S}}
\renewcommand{\S}{\mathcal S}

\renewcommand{\-}{\setminus}

\DeclareMathOperator{\dist}{dist}

\newcommand\HD{{\rm{HD}}} 
\newcommand{\fb}[1]{\partial\{#1>0\}}
\newcommand{\fbs}[1]{\partial_{\rm{sing}}\{#1>0\}}

\newcommand{\I}[1]{\chi_{\{#1>0\}}}
\newcommand\vs{\vskip1cm}

\numberwithin{equation}{section} 

\usepackage{tikz}
\usetikzlibrary{arrows,shapes,positioning}
\usetikzlibrary{decorations.markings}
\tikzstyle arrowstyle=[scale=1]
\tikzstyle directed=[postaction={decorate,decoration={markings,
    mark=at position .65 with {\arrow[arrowstyle]{stealth}}}}]
\tikzstyle reverse directed=[postaction={decorate,decoration={markings,
    mark=at position .65 with {\arrowreversed[arrowstyle]{stealth};}}}]
\usetikzlibrary{patterns}

\title[Nerve conduction problem]
{Structure of singularities in the nonlinear nerve conduction problem}

\author{Aram L. Karakhanyan}
\address{School of Mathematics, University of Edinburgh, Peter Guthrie Tait Road
Edinburgh, 
EH9 3FD}
\email{aram6k@gmail.com}
\urladdr{www.maths.ed.ac.uk/$\sim$aram}

\date{\today \,(Last Typeset)}

\begin{document}

\begin{abstract}
We give a characterization of the singular points of the free boundary 
$\p \{u>0\}$ 
for viscosity solutions of the nonlinear equation
\begin{equation}\label{eq:abst}
F(D^2 u)=-\chi_{\{u>0\}}, 
\end{equation}
where $F$ is a 
fully nonlinear elliptic operator and $\chi$ the characteristic function. 
The equation \eqref{eq:abst} models the propagation of a nerve impulse along an axon.

We analyze the structure of the free boundary $\fb u$ near the singular points where $u$ and $\na u$ vanish simultaneously. Our method uses the stratification approach developed in \cite{DK}.  

In particular,  when $n=2$ we show that near a rank-2 flat singular free boundary point $\fb u$ is 
a union of four $C^1$ arcs tangential to a pair of crossing lines.
Moreover,  if $F$ is linear then the singular set of $\fb u$ is the union of degenerate and rank-2 flat points. 

We also provide an application of the boundary Harnack principles to study the higher order flat degenerate 
points,  and show that if $\{u<0\}$ is a cone then the blow-ups of $u$ 
are homogeneous functions.
\end{abstract}

\maketitle

\section{Introduction}\label{sec1}
In this paper we study the free boundary problem 
\begin{equation}\label{eq:F}
F(D^2 u)=-\chi_{\{u>0\}}\quad\ \mbox{in}\ \Omega, 
\end{equation}
where $\Omega\subset \R^n$ is a given bounded domain with $C^{2, \alpha}$ boundary,  $\I u$ the characteristic function function of $\{u>0\}$, and
$F$ a convex   fully nonlinear elliptic operator satisfying some structural conditions. 
The partial differential equation \eqref{eq:F} appears in a model of the nerve impulse 
propagation 
\cite{Feroe}, \cite{Pauwelussen}, \cite{Rinzel}. 

It comes from the following linearized 
diffusion system of FitzHugh
\begin{equation}\label{FitzHugh}
\left\{
\begin{array}{lll}
u_t=r(x)\Delta u+F(u, \vec v), \\
\vec v_t=G(u, \vec v), 
\end{array}
\right.
\end{equation}
where $u(x, t)$ is the voltage across the nerve membrane at distance $x$ and time $t$, 
and components of $\vec v=(v^1, \dots, v^k)$ model the conductance of the membrane to various ions \cite{Feroe}.
Mckean suggested to consider $F(u, \vec v)=-u+\I u$ \cite{McKean}.
Due to the homogeneity of the equation  the linear term disappears after quadratic scaling,  so we neglect it.

The linearized steady state equation 
\begin{equation}\label{eq:Lap}
\Delta u=-\I u
\end{equation} 
also arises in a solid combustion model \cite{Monneau-W}, and the composite membrane problem, see \cite{Kenig}, also \cite{CK} for its variational formulation. 

\medskip 
A chief difficulty is to analyze the free boundary near singular points where both $u$ and $\na u$ vanish. 
The main technique used in \cite{Monneau-W}, \cite{Kenig}, \cite{CK} is a monotonicity formula, which is not 
available for the nonlinear equations. 
The aim of this paper is to use the boundary Harnack principles and anisotropic scalings 
to  develop a new approach to circumvent the  lack of the monotonicity formula 
and obtain some of  the main results from \cite{Monneau-W} and \cite{Lee-Sh} for 
the fully nonlinear case.
More precisely, in this paper we address the optimal regularity, degeneracy and the shape of the free boundary near the singular points.

\bigskip

One of the 
main results in \cite{Monneau-W} concerns the  
cross shaped singularities in $\R^2$. It follows from the classification of homogeneous solutions and an application of the monotonicity formula introduced in \cite{Monneau-W}. 
For nonlinear equations this method cannot be applied. 
We remark that the degenerate case (i.e. when $u(x)=o(|x-x_0|^2)$ near a free boundary point $x_0$) 
cannot be treated by the monotonicity formula introduced in \cite{Monneau-W} because
it does not provide any qualitative information about $u$,  see Proposition 5.1 \cite{Monneau-W}.
Another approach was recently used in \cite{ST}.

\bigskip
It is well known that the strong solutions of \eqref{eq:Lap} may not be $C^{1, 1}_{loc}$, see  \cite{Chemin} Proposition 5.3.1.
However, if $F=\Delta$ then $\na u$ is always log-Lipschitz continuous  \cite{Yudovich} Lemma 2.1. 
For general elliptic operators one can show that $\na u$ is $C^{\alpha}$ for every $\alpha\in (0, 1)$, \cite{CC-Fully}, see also Remark \ref{rem:concave} in Section \ref{sec:2}.

\bigskip 

Another approach, based on Harnack inequalities, 
had been developed by Tolksdorf to prove the existence of homogeneous solutions for $\Delta_p u =0$ in cones \cite{Tolksdorf} .
We employ this approach in Section \ref{sec:5}.

The problem \eqref{eq:Lap} has some resemblance with the 
classical obstacle problem \cite{Caff-cpde}. For the fully nonlinear operators
the obstacle  problem has been studied in \cite{Lee} for one phase and  in 
\cite{Xavier} for the thin obstacle.

\medskip 

The paper is organized as follows: 
In Section \ref{sec:2} we state some technical results.
In Section \ref{sec:3} we prove the existence of viscosity solutions using a penalization argument. 
We also show the existence of  a maximal solution and establish its  non-degeneracy.
Section \ref{sec:4} contains the proof of the following dichotomy: either the free boundary points 
are rank-2 flat or the solution has quadratic growth. 
As a consequence we show that if $n=2$ then near a rank-2 flat point the free boundary 
is a union of four  $C^1$ curves tangential to a pair of crossing lines. 
This is done in Section \ref{sec:6}.  
Finally, in Section \ref{sec:5} we use a boundary Harnack principle to prove the homogeneity of blow-ups near conical free boundary points.

\vs
\section{Technical results}\label{sec:2}

Throughout this paper $B_r(x)$ 
denotes the open ball of radius $r$ centered at $x\in \R^n$ and $B_r=B_r(0)$.
For continuous function $u$ we let $u=u^+-u^-$, $u^+=\max(0, u)$, $\Omega^+(u)=\{u>0\}, \Omega^-(u)=\{u<0\}$,  and $\fbs u$ is the 
singular subset of the free boundary $\fb u$,  where $u=|\na u|=0$. 

We shall make two standing assumptions on the operators under consideration.
To formulate them we let $\S$ be the space of $n\times n$ symmetric matrices and $\S^+({\lambda, \Lambda})$
positive definite symmetric matrices with eigenvalues bounded between two positive constants $\lambda$ and $\Lambda$. 
\begin{itemize}
\item[$\bf F1^\circ $] 
The operator $F:\S\subset \R^{n\times n}\to \R$ is uniformly elliptic, i.e. there are two positive constants 
$\lambda, \Lambda$ such that 
\begin{equation}\label{eq:ellip}
\lambda\|N\|\le F(M+N)-F(M)\le \Lambda\|N\|,  \quad M\in \S,
\end{equation}
for every nonnegative matrix $N$.
\item[$\bf F2^\circ $]
$F$ is  
smooth except the origin and 
homogeneous of degree one $F(tM)=tF(M), t\in \R$, and $F(0)=0$.
\end{itemize}
For smooth $F$ the hypothesis $\bf F1^\circ$ is equivalent to 
\[
\lambda \left| \xi \right| ^{2}\leq F_{ij}\left( M\right) \xi_i \xi _{j}\leq \Lambda \left| \xi \right| ^{2}, 
\]
where $F_{ij}(S)=\frac{\p F(S)}{\p s_{ij}}, S=[s_{ij}]$. 

Typically, $F(M)=\sup_{t\in \mathcal I}A_{ij, t}M_{ij}$, 
where $\mathcal I$ is the index set and $A_{ij,t}\in \S^+({\lambda, \Lambda})$
such that 
$\lambda \left| \xi \right| ^{2}\leq A_{ij, t}\xi_i \xi _{j}\leq \Lambda \left| \xi \right| ^{2}$.
Notice that if $w_t(x)=w(A_t^{\frac12}x)$ then 
$\Delta w_t=A_{ij, t}w_{ij}$.

\medskip 

We also define Pucci's extremal operators 
\[
\M^{-}\left( M,\lambda ,\Lambda \right) =\lambda \sum _{e_{i} > 0}e_{i}+\Lambda \sum _{e_{i} < 0}e_{i}, 
\quad 
\M^{+}\left( M,\lambda ,\Lambda \right) =\Lambda \sum _{e_{i} > 0}e_{i}+\lambda \sum _{e_{i} < 0}e_{i},
\]
where $e_1\le e_2\le \dots\le e_n$ are the eigenvalues of $M\in \S$.

\begin{definition}
A continuous function $u$ is said to be a viscosity solution  of 
$F(D^2u) = -\I u$, 
if the equation $F(D^2v(x_0)) = -\I u$ holds   pointwise, whenever 
at $(x_0, u(x_0))$ the graph of $u$ can be touched from above and below by paraboloids $v$.
\end{definition}
\begin{remark}\label{rem:concave}
If $F$ is concave and $u$ is a viscosity solution of $F(D^2 u)=0$ in $B_1$ then 
\begin{equation}
\| u\|_{C^{2, \alpha}(B_{\frac12})}\le C\left(\| u\|_{L^\infty(B_1)}+|F(0)|\right), 
\end{equation}
where $0<\alpha<1$ and $C$ are universal constants, see Theorem 6.6 \cite{CC-Fully}.
If $F$ is convex or concave then for the viscosity solutions of $F(D^2u)=0$ we still have the estimate 
\[
\|u\|_{C^{1,1}(B_{1/11})}\le C\|u\|_{L^\infty(B_1)},
\]
(see (6.14) and Remark 1 on page 60 in \cite{CC-Fully}). 
Theorem \ref{thm:only-conc} is the only place where we require $F$ to be convex. 
\end{remark}

Under assumptions $\bf F1^\circ-F2^\circ$ the classical weak and strong 
comparison principles are valid for the viscosity solutions \cite{CC-Fully}. Moreover, we have 
the strong and Hopf's comparison principles.

\begin{lem}[Strong comparison principle]\label{lem:SCP}
Suppose $v\in C^2(D), w\in C^1(D)$, and $\na v\not \equiv 0$ in a bounded domain $D$.
Let  $F(D^2 v)\ge 0\ge F(D^2 w)$ in 
 $D\subset \R^n$ in viscosity sense,   $v\le w$, and $v, w$ are not identical,  
then
\begin{equation}
v< w \ \text{in} \ D.
\end{equation}
\end{lem}
See Theorem 3.1 \cite{Giga}.

\begin{lem}[Hopf's comparison principle]\label{lem:HCP}
Let $B$ be a ball contained in $D$ and assume that $w \in C^1(D) , v \in C^2(D)$ and that $\na v \not = 0 $, in $B$. 
Let $v$ and $w$ be a viscosity subsolution and a supersolution of $F(D^2 u)=0$, respectively.
Moreover, suppose that $v <w$, in $B$, and that $v(x_0) = w(x_0)$, for some $x_0 \in \p B$. 
Then, $\na v(x_0) \not =\na w(x_0)$. 
\end{lem}
See Theorem 4.1 \cite{Giga}.

One of the main tools in our analysis is the boundary Harnack principle.
As before,  we assume that $F$ is smooth, homogeneous of degree 1, uniformly elliptic with ellipticity constants $\lambda$ and 
$\Lambda$, and $F (0) = 0$. 
We use the following notation: 
$f(x'), x'\in B_1'\subset \R^{n-1}$ is Lipschitz continuous 
function with Lipschitz constant $M>1, f(0)=0$, 
$\Omega_r=B_r'\times [-rM, rM]\cap \{x_n>f(x')\}$, 
$\Delta_r=B_r'\times [-rM, rM]\cap \{x_n=f(x')\}$, 
$A=e_nM/2$.

Then we have the following Harnack principle, see \cite{Wang}. 
\begin{thm}
Assume $\bf F1^\circ-F2^\circ$ hold and $F$ is either concave or convex. Let $u, v$ be two nonnegative solution of 
$F(D^2 u)=0$ in $\Omega_1$
that equal 0 along $\Omega_1\setminus S$.
Suppose also that $u-\sigma v\ge 0$ in $\Omega_1$
for some $\sigma\ge 0$.
Then for some constant $C$, depending only on 
$\lambda, \Lambda, n$ and the Lipschitz character of $\Omega_1$, we have in 
$\Omega_{\frac12}$
\begin{equation}\label{eq:bHarnack}
C^{-1}\frac{u(A)-\sigma v(A)}{v(A)}\le
\frac{u-\sigma v}{v}
\le
C\frac{u(A)-\sigma v(A)}{v(A)}.
\end{equation}
\end{thm}
Furthermore, as in \cite{C-1} (see also \cite{Wang} Section 2) one can show that 
the nonnegative  solutions in $\Omega_1$ are monotone in $\Omega_{\delta_0}$
for some universal $\delta_0$. We state this only in two spatial dimensions.

\begin{thm}\label{C-mon}
Let $w$ be a viscosity solution of $F(D^2w)=0, w\ge 0$ in $D=\{|x_1|\le 1, f(x_1)<x_2\le M\}$, $M=\|f\|_{C^{0, 1}}$.
Assume $\bf F1^\circ-F2^\circ$ hold and $F$ is either concave or convex. 
Then there is $\delta=\delta(M)$ such that 
\[
\p_{2} w\ge 0 \quad \mbox{in} \ D_\delta=\{|x_1|\le \delta, f(x_1)<x_2\le M\delta\}. 
\]
\end{thm}
In \cite{Wang} Theorem \ref{C-mon} is stated for concave operator $F$, however the concavity 
is needed only to assure that locally the viscosity solutions of the homogeneous equation 
are locally $C^{2, \alpha}$ regular, see Remark 1.2 in \cite{Wang}. Since in the proofs of Lemmata 2.1-2.5 in \cite{Wang} one needs only 
$C^{1, \alpha}$ regularity of the solutions then in view of Remark \ref{rem:concave} we see that 
Theorem \ref{C-mon} continues to hold for convex $F$, see \cite{Feldman}.

Finally, we give a characterization of homogeneity, see Theorem 2.1.1 \cite{Tolksdorf} for a proof. 
\begin{lem}\label{lem:hom-Tolk}
Let $w\in L^\infty(\K_1)\- \{0\}$ where $\K_1=\K(0, 1)$ such that 
\begin{equation}\label{eq:hom-Tolk}
C_Rw(Rx)=w(x), \quad R\in(0,1), \  \  C_R=\frac1{\sup_{\K_R}w}.
\end{equation}
Then there is a $\kappa\in [0, \infty)$ such that 
\begin{equation}
w(x)=|x|^\kappa w\left(\frac x{|x|}\right).
\end{equation}

\end{lem}


\vs
\section{Existence and nongeneracy}\label{sec:3}
In this section we prove the existence of viscosity solutions and the non-degeneracy of maximal solutions. 

\subsection{Existence of viscosity solutions}

\begin{definition}
A continuous function $u$ is said to be a viscosity subsolution  of 
$F(D^2u) = -\I u$, 
if the inequality $F(D^2v(x_0)) \ge -\I u$ holds  pointwise, whenever 
at $(x_0, u(x_0))$ the graph of $u$ can be touched from below by a paraboloid $v$.
Moreover, $u$ is said to be a strict subsolution if the inequality above is strict. 
\end{definition}

\begin{definition}
A viscosity solution $u$ of $F(D^2 u)=-\I u$ is said to be maximal in $D$ if 
for every strong subsolution $v$ satisfying $v\le u$ on 
$\partial D'$ for some  subdomain $D'\subset D$ we have $v\le u$ in $D'$. 
\end{definition}

\begin{thm}
Assume $\bf F1^\circ-F2^\circ$ hold. Let $D$ be a bounded $C^{2, \alpha}$  domain and $g\in C^{2, \alpha}(\overline D)$.
There exists a viscosity solution $u$ to  
\begin{equation}
\left\{
\begin{array}{lll}
F(D^2 u) =-\I u \quad &\mbox{in}\ D, \\
u=g &\mbox{on}\ \partial D, 
\end{array}
\right.
\end{equation} 
such that $u\in W^{2, p}(D)$ for every $p\ge1$.  
\end{thm}
\begin{proof}

\medskip

We use a standard penalization argument \cite{Friedman}. 
Let $\beta_\e(t), t\in \R$ be a family of $C^\infty$ 
functions 
such that 
\begin{equation}\label{eq:beta-0}
\left\{
\begin{array}{lll}
 \beta_\e(t)\ge \I t\quad &\mbox{on}\ \R, \\
\beta_{\e'}(t)\le \beta_{\e}(t)& \mbox{if}\  \e'<\e,\\
\lim_{\e\to 0}\beta_\e(t)= \I t &t\in \R.\\
\end{array}
\right.
\end{equation}
Given $\e>0$,  there is a solution $v$ of 
\begin{equation}\label{eq:e}
\left\{
\begin{array}{lll}
F(D^2 v)=-\beta_\e(v) \quad &\mbox{in}\ D, \\
v=g &\mbox{on}\ \partial D.
\end{array}
\right.
\end{equation} 
Observe that Perron's method implies that for 
every $\e>0$ the maximal solution $u_\e$ exists. Furthermore, 
since $\beta_\e$ are uniformly bounded then $\| v\|_{W^{2, p}(D)}\le C$ 
with some $C$ independent of $\e$, see  Theorem 7.1 \cite{CC-Fully} and Remark \ref{rem:concave} above.  

If $v$ is a subsolution, i.e. 
$F(D^2 v)\ge -\I v$ then by \eqref{eq:beta-0} we also have  $F(D^2 v)\ge -\beta_\e(v)$.
Thus for $\e>\e'$  (using \eqref{eq:beta-0}) we get
\[
F(D^2 u_{\e'})=-\beta_{\e'}(u_{\e'})\ge -\beta_{\e}(u_{\e'}).
\]
This shows that $u_{\e'}$ is a subsolution to \eqref{eq:e}.
Since $u_\e$ is the maximal solution then we have 
\[
v\le u_\e, \quad u_{\e'}\le u_\e.
\]
Thus $u(x)=\lim_{\e\to 0}u_{\e}$ in $W^{2, p}$.
From the uniform convergence it follows that $u(z)>0$ implying that 
$u_{\e}>0$ in some neighborhood of $z$. Thus 
$F(D^2 u)=-1$ near $z$.
Since $D^2 u=0$ a.e. on $\{u=0\}$ it follows that $F(D^2 u)=-\I u$. 
\end{proof}


\subsection{Non-degeneracy}

\begin{thm}\label{thm:nondeg}
Assume $\bf F1^\circ-F2^\circ$ hold. Let $u$ be the maximal solution,  then 
there is a universal constant $c_{n, \gamma}$, 
depending only on dimension $n$ and $\gamma=\frac{\Lambda(n-1)}{\lambda}-1$, such that 
\[
\inf_{B_r(x_0)} u>-c_{n, \gamma}r^2
\]
implies that $u(x_0)>0$. 
\end{thm}
\begin{proof}
Let us consider 
\[
b(x)=
\left\{
\begin{array}{lll}
C(1-|x|^2) & \mbox{if}\ |x|\le 1, \\
\phi(x)-\phi(1) & \mbox{if}\ |x|> 1, 
\end{array}
\right.
\]
where 
\[
\phi(x)=
\left\{
\begin{array}{lll}
-\log |x| & \mbox{if}\ n=2, \\
\frac1{\gamma}{|x|^{- \gamma}} & \mbox{if}\ n\ge 3, 
\end{array}
\right.
\]
and the constant $C$ is chosen so that $b(x)$ is $C^1$ regular. 
It is straightforward to compute $D^2 b$ and thus  
\[
F(D^2 b)=
\left\{
\begin{array}{lll}
-2C F(\delta_{ij}) & \mbox{if}\ |x|\le 1,\\
-\frac1{|x|^{\gamma+2}} F(\delta_{ij}-(\gamma+2)\frac{x_ix_j}{|x|^2})& \mbox{if}\ |x|>1.
\end{array}
\right.
\]
From the ellipticity \eqref{eq:ellip} we get that 

\[
F(\delta_{ij}-(\gamma+2)\frac{x_ix_j}{|x|^2})\le 
\mathcal M^+ (\delta_{ij}-(\gamma+2)\frac{x_ix_j}{|x|^2})=0, \quad |x|>1.
\]
Hence 
\[
F(D^2 b)\ge -\frac1{|x|^{\gamma+2}} \mathcal M^+(\delta_{ij}-(\gamma+2)\frac{x_ix_j}{|x|^2})=0, \quad  |x|>1.
\]
Consequently, we see that  $\hat b(x)=\frac{ b(x)}{2C F(\delta_{ij})}$ is a subsolution. 

Given  $r$, choose $\rho$ so that $\frac 2\rho=r$. 
Then 
for $|x|>\frac1\rho$ we have 
\[
\frac1{\rho^2} \hat b({\rho}x)=\frac1{\rho^{\gamma+2} \gamma} \left[ \frac1{|x|^{\gamma}}-\rho^{\gamma}\right], 
\]
and consequently 
\begin{eqnarray*}
\frac{\hat b(r)}{\rho^2}
&=&
\frac1{\rho^2} \hat b(\frac2{\rho})=-\left(1-\frac1{2^{\gamma}}\right)\frac1{\rho^2\gamma}\\
&=&
-\left(1-\frac1{2^{n-2}}\right)r^2\frac1{4\gamma}=:-c_{n, \gamma} r^2.
\end{eqnarray*}
Thus $u(0)\ge \hat b(0)>0$.

\end{proof}
\vs
\section{Dichotomy}\label{sec:4}
In order to formulate the main result of this section we first 
introduce the notion of rank-2 flatness. 
Let $P_2$ be the set of all homogeneous normalized polynomials of degree two, i.e.
\begin{equation}\label{poliu7u65}
P_2:=\left\{p(x)=\sum 
a_{ij}x_ix_j, {\mbox{ for any }}
x\in \R^n, {\mbox{ with }} \|p\|_{L^\infty(B_1)}=1 \right\},
\end{equation}
where $a_{ij}$ is a symmetric $n\times n$ matrix. For given~$p\in P_2$
and~$x_0\in\R^n$, we set~$p_{x_0}(x):=p(x-x_0)$ and consider the 
zero level set of translated polynomial $p$
\begin{equation}\label{spx}
S(p, x_0):=\{x\in \R^n : p_{x_0}(x)=0\}.\end{equation}
By definition $S(p, x_0)$ is a cone with vertex at~$x_0$. 

\begin{definition}\label{def1}
Let~$\delta>0$,~$R>0$ and~$x_0\in \fb u$.
We say that~$\fb u$ is $(\delta, R)$-rank-2 flat at~$x_0$
if, for every~$r\in (0, R]$, 
there exists~$p\in P_2$ such that
\[
\HD\Big(\fb u\cap B_r(x_0), S (p, {x_0})\cap B_r(x_0)\Big)< \delta\, r.
\]\end{definition}
Here $\HD$ denotes the Hausdorff distance
defined as follows
\begin{equation}\label{defHD}
\HD(A, B):=\max\left\{\sup\limits_{a\in A}\dist(a, B),\; 
\sup\limits_{b\in B}\dist(b, A)\right\}.
\end{equation}
Given~$r>0$, $x_0\in\partial\{u>0\}$ and~$p\in P_2$, we let
\begin{equation}\label{defflat:BIS} h_{\rm{min}}(r,x_0,p):=
\HD\Big(\fb u\cap B_r(x_0), S(p, x_0)\cap B_r(x_0)\Big).
\end{equation}
Then, we define the rank-2
flatness at level $r>0$ of $\fb u$ at~$x_0$ as follows.

\begin{definition}\label{def:flat}
Let~$\delta>0$,~$r>0$ and~$x_0\in\partial\{u>0\}$.
We say that $\fb u$ is $\delta$-rank-2 flat at level $r$ at~$x_0$
if~$h(r, x_0)<\delta r$, where 
\begin{equation}\label{flatdef}
h(r, x_0):=\inf_{p\in P_2}h_{\rm{min}}(r,x_0, p).
\end{equation}
\end{definition}

In view of Definitions~\ref{def1} and~\ref{def:flat}, we can say that~$
\fb u$ is $(\delta, R)$-rank-2 flat at~$x_0\in\partial\{u>0\}$
if and only if, for every~$r\in (0, R]$, it is 
$\delta$-rank-2 flat at level $r$ at~$x_0$.

\begin{thm}\label{growth}
Let $n\ge 2$ and $u$ be a viscosity solution of ~\eqref{eq:F}.
Let~$D\subset\Omega$, $\delta>0$ and let~$x_0\in \fb u\cap D$
such that~$|\na u(x_0)|=0$ and~$\fb u$
is not~$\delta$-rank-2 flat at~$x_0$ at any level~$r>0$.
Then, $u$ has at most quadratic growth at~$x_0$, bounded from above
in dependence on~$\delta$.
\end{thm}

Theorem \ref{growth} will follow from Proposition \ref{zimbo} below  in standard way, see \cite{DK}.
Let us define $r_{k}=2^{-k}$ and 
$M\left( r_k,x_{0}\right) =\sup _{B_{r_{k}}\left( x_0\right) }\left| u\right|,$
where $x_{0}\in \partial \left\{ u > 0\right\} \cap \left\{ \left| \nabla u\right| =0\right\} $.

\begin{prop}\label{zimbo}
Let $u$ be as in Theorem \ref{growth} and $\sup|u|\le 1$. If 
\[
h(r_k, x_0, )>\delta r_k
\]
for some $\delta>0$, 
then 
there exists $C=C(\delta, n, \lambda, \Lambda)$ such that 
\begin{equation}\label{eq:M-ineq}
M\left( r_{k+1},x\right) 
\leq 
\max \left( 
Cr^{2}_{k},\dfrac {1}{2^{2}}M\left( r_k,x\right) ,\dots, 
\dfrac {M\left( r_{k-m},x\right) }{2^{2\left( m+1\right) }},\dots, 
\dfrac {M\left( r_0,x\right) }{2^{2\left( k+1\right) }}
\right). 
\end{equation}
\end{prop}
\begin{proof}

If \eqref{eq:M-ineq} fails then there are solutions $\{u_j\}$ of \eqref{eq:F} with $\sup|u_j|\le 1$, sequences $\{k_j\}$ of integers, and free boundary points $\{x_j\}, x_j\in B_1$
such that 
\begin{equation}\label{eq:M-ineq-fail}
M\left( r_{k_j}+1,x_j\right) 
> 
\max \left( 
j r ^{2}_{k_j},\dfrac {1}{2^{2}}M\left( r_{k_j},x_j\right) ,\dots, 
\dfrac {M\left( r_{k_j-m},x_j\right) }{2^{2\left( m+1\right) }},\dots, 
\dfrac {M\left( r_0,x_j\right) }{2^{2\left( k_j+1\right) }}
\right),  
\end{equation}
where with some abuse of notation we set $M\left( r_{k_j},x_j\right)=\sup_{B_{r_{k_j}}(x_j)}|u_j|$. 
Since $M(r_{k_j}, x_j)\le \sup_{B_1} |u_j|<\infty$ it follows that $k_j\to \infty$. 
Define the scaled functions 
\[
v_{j}\left( x\right) =\dfrac {u_{j}\left( x_{j}+r_{k_j}x\right) }{M\left( r_{k_{j}}+1,x_{j}\right) }.
\]
By construction we have 
\begin{eqnarray*}
v_{j}\left( 0\right) =0, \quad \left| \nabla v_{j}\left( x\right) \right| =0,\\
\sup_{B_{\frac12}}|v_j|=1, \\
h(0, 1)>\delta, \\
v_j(x)\le 2^{2{m-1}}, |x|\le 2^m, m<2^{k_j}, 
\end{eqnarray*}
where the last inequality follows from \eqref{eq:M-ineq-fail} after rescaling the 
inequality $\frac{M(r_{k_j-m}, x_j)}{M(r_{k_j+1}, x_j)}<2^{2(m-1)}$.  
Utilizing the homogeneity of  operator $F$ and noting that 
$$
D^2_{x_\alpha x_\beta} v_j(x)=r_j^2 \left(D^2_{\alpha\beta} u_j\right)(x_j+r_{k_j}x), 
$$ 
it follows that 
		\begin{equation}
		F(D^2 v_j(x))=-\frac{r_{k_j}^2}{M(r_{k_j+1}, x_j)}\I{v_j}=-\sigma_j \I{v_j}, 
		\end{equation}
where $\sigma_j=\frac{r_{k_j}^2}{M(r_{k_j+1}, x_j)}$. 
Observe that $\sigma_j<\frac1j$ in view of 
\eqref{eq:M-ineq-fail}.
Since under hypotheses $\bf  F1^\circ-F2^\circ$ we have local $W^{2, p}$ bounds for all $p\ge 1$ 
(see Theorem 7.1 \cite{CC-Fully})
it follows that we can employ a customary compactness argument for the viscosity solutions
to show that there is a function $v_0\in W^{2, p}_{loc}(\R^n)$ such that 
\begin{eqnarray}\nonumber
v_{k_j}\to v_0 \quad \mbox{in}\ C^{1, \alpha}_{loc}(\R^n), \\\nonumber
v_0(0)=|\na v_0(0)|=0, \\\label{flat-fail}
h(0, 1)>\delta, \\\nonumber
F(D^2 v_0)=0. 
\end{eqnarray}
From Liouville's theorem it follows that $v_0$ is homogeneous quadratic polynomial of degree two.
This is in contradiction with \eqref{flat-fail} and the proof is complete.
\end{proof}

\begin{remark}

In \cite{Lee-Sh} the authors proved some partial results for the 
problem 
\begin{equation}\label{potato}
F(D^2 u)=\chi_{\mathcal D}\  \mbox{in}\ B_1,  \quad u=|\na u|=0\ \mbox{in}\  B_1\setminus \mathcal D.
\end{equation}
For  $F=\Delta$ this problem arises in the linear potential theory related to harmonic continuation of the Newtonian potential of $B_1\cap \mathcal D$.

Analysis similar to that in the proof of Proposition \ref{zimbo} shows
that the result is also valid for the solutions of \eqref{potato}.
\begin{cor}\label{cor-Lee}
Let $u$ be a viscosity solution to \eqref{potato}, 
 then the statement of Theorem \ref{growth} holds for $u$ too.
\end{cor}
\end{remark}

\vs
\section{Quadruple junctions}\label{sec:6}
Throughout this section we assume that $F$ is convex, satisfies $\bf F1^\circ-F2^\circ$ and $u$ is a  
viscosity solution, see Section \ref{sec:3}. 
\begin{lem}\label{lem:mon}
Assume $\bf F1^\circ-F2^\circ$ hold and $F$ is convex. Let $n=2$ and $0\in \fb u, |\na u(0)|=0$ be a rank-2 flat point such that  the zero set of the polynomial $p(x)=M^2 x_1^2-x_2^2, M>0$
approximates $\fb u$ near $0$. 
Assume further that $u$ is non-degenerate at $0$. 
Then 
for every $\delta_0>0$ there is $r_0=2^{-k_0}$ (for some $k_0\in \mathbb N$) 
such that $\p_2 u(x+te)\ge 0$ whenever  $x\in (B_{r_0}\setminus B_{\delta_0r_0})\cap K^-$
and $\delta_0\le t\le 2$. 
\end{lem}
\begin{proof}
Let $\theta_0=\arctan M$ and denote $K^-=\{x_2\ge M|x_1|\}$.
After rotation of coordinate system we can assume that $K^-$ contains 
$u<0$ away from some small neighborhood of 
$x_2=M|x_1|$ (the green cones in Figure \ref{fig1} represent that neighborhood).

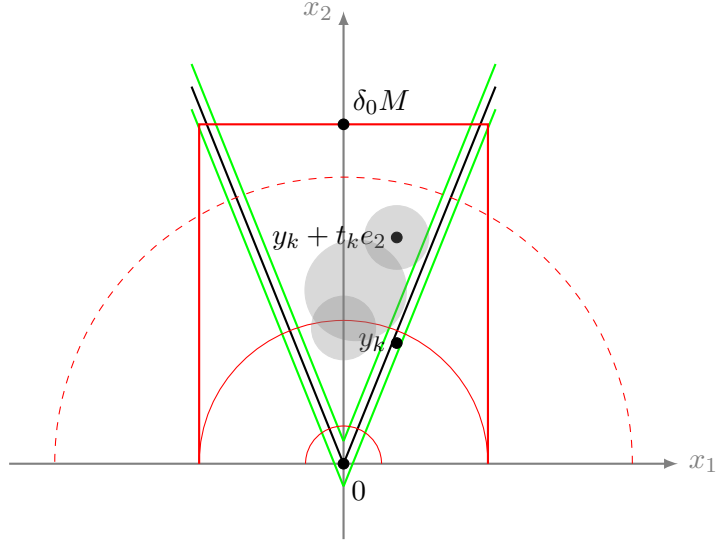
\begin{figure}
\begin{center}

\begin{tikzpicture}
 \draw[gray, thick, -latex] (-4.4, 0) -- (4.4, 0) node[right]{$x_1$}; 
  \draw[gray, thick, -latex] (0, -1) -- (0, 6) node[left]{$x_2$}; 
    \draw[thick] (0, 0) -- (2, 5) ; 
        \draw[thick] (0, 0) -- (-2, 5) ; 
        \draw[green, thick] (0, 0.3) -- (2, 5.3) ; 
        \draw[green, thick] (0, 0.3) -- (-2, 5.3) ; 
         \draw[green, thick] (0, -0.3) -- (2, 4.7) ; 
        \draw[green, thick] (0, -0.3) -- (-2, 4.7) ; 
                \draw[red, thick] (-1.9, 0) -- (-1.9, 4.5) -- (1.9, 4.5) -- (1.9,0); 
                \draw[red] (0.5, 0)  arc[radius = 5mm, start angle= 0, end angle= 180] ;
                                \draw[red] (1.9, 0)  arc[radius = 19mm, start angle= 0, end angle= 180];
                \draw[red, dashed] (3.8, 0)  arc[radius = 38mm, start angle= 0, end angle= 180] ;
  \filldraw[black] (0.7,1.6) circle (2pt) node[left] {$y_k$}; 
    \filldraw[black] (0.7,3) circle (2pt) node[left] {$y_k+t_k e_2$}; 
  \filldraw[gray, opacity=0.3] (0.7,3) circle (12pt);
  \filldraw[gray, opacity=0.3] (0.16,2.3) circle (19pt);
  \filldraw[gray, opacity=0.3] (0,1.8) circle (12pt);

  \filldraw[black] (0,4.5) circle (2pt);
   \node[left, above] at (0.5, 4.5) {$\delta_0M$};
  \filldraw[black] (0,0) circle (2pt); 
  \node[below] at (0.2, -0.1) {$0$};

\end{tikzpicture}

\end{center}

\caption{The geometric construction in the proof of Lemma \ref{lem:mon}.  
The shadowed balls are in the Harnack chain.}
\label{fig1}
\end{figure}


Suppose the claim fails, then there is $\delta_0>0$ so that for every 
$r_k=2^{-k}\to 0$ and some points $x_k\in B_{r_k}\setminus B_{\delta_0r_k}\cap \Omega^-(u)$
we have 
\begin{equation}\label{cod-0}
\p_2u(x_k+r_kt_k e_2)<0, \quad \mbox{for some }\ \delta_0\le t_k\le 2.
\end{equation}
We can choose $\delta_0$ so that for large $k$ there holds $\delta_0>\frac{h_k}{\cos\theta_0}\to 0, $
where $h_k=h(2^{-k}, 0)$.

Introduce the scaled functions 
\begin{equation}
v_k(x)=
\left\{
\begin{array}{lll}
\displaystyle
\frac{u(r_k x)}{M(r_{k})}\quad &\mbox{if} \ \eqref{eq:M-ineq}\ \mbox{ is true for all}\  k\ge \hat k,  \mbox{for some fixed}\ \hat k, \\
\displaystyle
\frac{u(r_k x)}{M(r_{k+1})}\quad &\mbox{if there is a sequence } r_k=2^{-k} \ \mbox{such that }\ \eqref{eq:M-ineq-fail} \ \mbox{holds}.\\
\end{array}
\right.
\end{equation}
Here we set $M(r_k)=M(r_k, 0)$. 
For both scalings  we have that $v_k$'s are non-degenerate: for the first scaling  it follows from Theorem \ref{thm:nondeg} 
(our assuption on non-degeneracy), for the second one  it follows that 
$\sup_{B_{1/2}}|v_k|=1$.

Moreover, by \eqref{cod-0} there is $y_k\in (B_1\setminus B_{\delta_0})\cap \{v_k<0\}$ such that 
\begin{equation}\label{}
\p_2 v_k(y_k+ t_k e_2)<0, \quad \mbox{for some }\ \delta_0\le t_k\le 2.
\end{equation}

There is a subsequence $y_{k_j}+t_{k_j}e_2\to y_0+t_0e_2\in K^-\cap B_2, $ 
there is a Harnack chain $B^1, \dots, B^N$ where $B^1=B_{\cos\theta_0/2}(e_2)$
and $B^N=B_{\delta_0/2}(y_0)$, $N$ is independent of $k_j$.
Let $\widetilde K=B_1\cup\bigcup_{i=1}^N B^i$.
Since under hypotheses $\bf  F1^\circ-F2^\circ$ we have local $W^{2, p}$ bounds for all $p\ge 1$ 
(see Theorem 7.1 \cite{CC-Fully})
it follows that we can employ a customary compactness argument for viscosity solutions
to infer that there is a function $v_0\in W^{2, p}_{loc}(\R^n)$ such that
we have 
\begin{eqnarray*}
v_k\to v_0 \quad &\mbox{in}\ W^{2, p}, \forall p\ge 1, \\
|F(D^2v_k)|\le C &\mbox{uniformly}, \\
v_0<0 &\mbox{in}\ K^-, \\
|v_k|\le C &\mbox{in Harnack chain domain}\ \widetilde K, \\
\p_2v_0(y_0+t_0e_2)\le 0. 
\end{eqnarray*}
From Theorem \ref{C-mon} it follows that 
$\p_2v_0(y_0+t_0e_2)=0$. Moreover, $w=\p_2 v_0$ satisfies the equation 
$F_{ij} D_{ij} w=0$ in $\widetilde K$, hence from 
the strong maximum principle it follows that $w=0$ in $K^-$. 
Consequently, 
$v_0$ depends only on $x_1$ implying that 
$\theta_0=0$ or $\theta_0=\pi/2$ which is a contradiction.

\end{proof}

\begin{thm}\label{thm:only-conc}
Let $u$ be as in Lemma \ref{lem:mon}.
Then in some neighbourhood of $0$
the free boundary consists of four  $C^1$ curves tangential to the zero set   of the polynomial $M^2 x_1^2-x_2^2$.
\end{thm}

\begin{proof}
Let $\fbs u=\fb u\cap \{|\na u|=0\}$. 
Clearly, it is enough to prove that there is $r$ such that $\fbs u\cap B_r=\{0\}$.
Suppose the claim fails. Then there is a sequence $x_k\in \fbs u$, $x_k\to 0$. Let 
$M_k^-:=M^-(2r_k\ell_0)=\sup_{B_{2r_k\ell_0}} u^-, r_k=|x_k|$ and consider 
\begin{equation}
v_k(x)=\frac{u(r_k x)}{M^-(2r_k\ell_0)} \ \text{where} \ \ell_0=\sqrt{\frac\Lambda \lambda}.
\end{equation}
Note that $F(D^2 v_k)=-\I{v_k}\frac{r_k^2}{M^-(2r_k\ell_0)}$, and therefore 
by dichotomy (see Section \ref{sec:4}) and non-degeneracy $|F(D^2 v_k)|\le C$ for some $C>0$ independent of $k$.

By construction $\sup_{B_{2\ell_0}}|v_k^-|=1$ and since $F(D^2 v_k)=0$ in $\Omega^-(v_k):=\{v_k<0\}$ it 
follows that there is $z_k\in \p B_{2\ell_0}\cap \Omega^-(v_k)$ such that $v_k^-(z_k)=1$. Consequently, 
$\dist(z_k+\delta_0 e_2, \{p=0\})\ge \delta_0/2$ and by Lemma \ref{lem:mon}
\[
v_k^-(z_k+\delta_0 e_2)\ge 1.
\]
\begin{claim}
With the notation above we have 
\[
M^+_k\le M_k^-.
\]
\end{claim}
To check this we first observe that $\text{trace}(A_t D^2 v_k(x))\le F(D^2 v_k x)$ thanks to the 
convexity  
 and 
$A_t\in \mathcal S_{\lambda, \Lambda}$. Now  consider $w_{k, t}(x)=v_k(A^{\frac12}_t x)$ then 
$\Delta w_{k, t}(x)=\text{trace}(A_t D^2 v_k(x))\le F(D^2 v_k (x))$. Since $w_{k, t}$
 is continuous and $w_{k, t}(0)=0$ then one can easily check that 
\begin{equation}\label{2fgnumb}
\fint_{B_r}w_{k,t}=\int_0^r\frac1{t}\int_{B_t}\Delta w_{k, t}\le 0
\end{equation}
because 
of  convexity of $F$ and the estimate $F(D^2 v_k)\le 0$.

Note that 
\[
\int_{B_r}w_{k,t}(x)dx=\frac1{\sqrt{\det A_{t}}}\int_{\left|A^{-\frac12}y\right|<r}v_k(y)dy\le 0.
\]

Thus from \eqref{2fgnumb} it follows that 
\begin{eqnarray*}
\frac1{\Lambda}\int_{B_{\frac r{\sqrt\Lambda}}}v_k^+
&\le& 
\frac1{\sqrt{\det A_{t}}}\int_{\left|A^{-\frac12}y\right|<r}v_k^+(y)dy\le \frac1{\sqrt{\det A_{t}}}\int_{\left|A^{-\frac12}y\right|<r}v_k^-(y)dy\\
&\le &
\frac1{{\lambda }}\int_{B_{\frac r{\sqrt \lambda}}}v_k^-(y)dy, \ \  r<2.
\end{eqnarray*}
Consequently,  we get
\[
\int_{B_r}v_k^+(y)dy\le \ell_0^2 \int_{B_{r\ell_0}}v_k^-. 
\]

Let $\widehat v_k=v_k+\widehat C|x|^2$ 
then  
\[
F(D^2 \widehat v_k)\ge F(D^2v_k)+2\widehat C\lambda \ge 0,
\]
provided that $\widehat C$ is sufficiently large. 

We see that $\widehat v_k$ is a subsolution, and hence so is $\widehat v_k^+$.
Consequently, applying the weak Harnack inequality \cite{CC-Fully} we get 
\[
\sup_{B_{\frac43}}v_k^+\le \sup_{B_{\frac43}}\widehat v_k^+\le c_0\int_{B_{2}}(v_k^-+\widehat C|x|^2)\le c_0(1+2\pi \widehat C). 
\]
This completes the proof of the claim.

Thus, as in the proof of Lemma \ref{lem:mon}, we can employ a customary compactness argument 
in $W^{2, p}$  so that $y_k=x_k/r_k\to y_0\in \{x_2=M|x_1|\}\cap 
\p B_1$ and 
\[
\na v_0(y_0)=0, \quad v_0(z_0+\delta_0e_2)\ge 1
\]
by Harnack chain and $C^{1, \alpha}$ estimates in the Harnack chain domain (which joins $2\ell_0e_2$ with $z_0+\delta_0e_2$).
Since at $y_0$ free boundary is a line we can apply Hopf's lemma to conclude that 
$v_0^-\equiv 0$ which is a contradiction.  
\end{proof}

\vs
\section{Existence of homogenous blow-ups}\label{sec:5}
In this section we show that if the free boundary is a cone then one can 
blow-up the solution at the vertex so that the limit is  a homogeneous function.
We start with a doubling inequality which provides a bound for the rate of the scaling at the vertex. 
\begin{lem}\label{lem:doubling}
Assume $\bf F1^\circ-F2^\circ$ hold and $F$ is concave. Let $\K_t=\{x=r\sigma, \sigma\in S, 0< r<t\}$, $S\subset \sf^{n-1}$ such that 
$\p S$ is smooth. Let $v$ be the solution to the 
Dirichlet problem $F(D^2 v)=0$ in $\K_1 $ and $v=v_0$ on $\p \K_1$ where 
\[
v_0(x)=
\begin{cases}
0 & x\in \p \K_1\cap B_{\frac34}, \\
16\left(|x|-\frac34\right)^2  &x\in \p \K_1\setminus B_{\frac34}.
\end{cases}
\]
Then there is a constant $\e>0$ such that 
\begin{equation}\label{eq:mon-R}
v(Rx)\le (1-\e(1-R))v(x), \quad x\in \K_1, \forall R\in\left[\frac12, 1\right].
\end{equation}
\end{lem}

\begin{proof}
Existence of $v$ follows from the Perron's method \cite{CC-Fully}.

Consider the barrier 
$b\left( x\right) =1+\dfrac {1}{\alpha ^{2}}\left( e^{-\alpha }-e^{-\alpha \left| x\right| ^{2}}\right)$ for some 
$\alpha>0$ to be fixed below. 
We have 
\begin{eqnarray}
b_{i}(x)
=
\dfrac {2x_i}{\alpha }e^{-\alpha \left| x\right| ^{2}},\quad  
D^{2}_{ij}b\left( x\right) 
=
\dfrac {2}{\alpha }e^{-\alpha \left| x\right| ^{2}}\left( \delta _{ij}-2\alpha x_{i}x_{j}\right) , \quad 1\le i, j\le n. 
\end{eqnarray}
Consequently, in the ring $\frac12\le |x|\le 1$ we have from $\bf F1^\circ$ and $\bf F2^\circ$ that  
\begin{eqnarray*}
F(D^2b(x))
&=&
\frac2\alpha e^{-\alpha|x|^2}F(\delta_{ij}-2\alpha x_ix_j)\\
&\le&
\frac2\alpha e^{-\alpha|x|^2}[F(\delta_{ij})-\lambda F(2\alpha x_ix_j)]\\
&\le& 
\frac2\alpha e^{-\alpha|x|^2}\lambda\left[\frac\Lambda\lambda-2\alpha|x|^2\right]\\
&\le& 0
\end{eqnarray*}
provided that $\alpha\ge 2\frac\Lambda\lambda$.

Since $b$ is concave function of $r=|x|$ and $v_0$ is convex in $r$ such that 
$b(0)\ge v_0(0)$ and $b(x)\ge v(x), |x|=1$,  then it follows that 
$b(x)\ge v_0(x)$ on $\p(\K_1\setminus K_{\frac12})$. 
Moreover, by the maximum principle we have $v(x)\le 1$ on $\p K_{\frac12}\cap \p B_{\frac12}$.
Thus $v(x)\le b(x)$ on   $\p(\K_1\setminus K_{\frac12})$, which in conjunction with 
$F(D^2 b(x))\le 0, \frac12\le |x|\le 1$ and the comparison principle implies that $v(x)\le b(x)$ in $(\K_1\setminus K_{\frac12})$.

\smallskip 

Furthermore, for $x\in \p\K_1\-S$ we have $v(Rx)=0$,  hence 
\eqref{eq:mon-R} is true for every $x\in \p\K_1\-S$. 
On the other hand on $S$ (where $|x|=1$)
we have 
\begin{eqnarray*}
v(Rx)\le b(Rx)
&=&1+ \frac1{\alpha^2}\left[e^{-\alpha }-e^{-\alpha R^2}\right]
\\
&\le&
1-\frac1\alpha e^{-\alpha}(1-R)\\
&=&
\left(1-\frac1\alpha e^{-\alpha}(1-R)\right)v(x), 
\end{eqnarray*}
where the last line  follows from the mean value theorem. Thus  \eqref{eq:mon-R} holds on $\p \K_1$ 
with $\e=\frac{e^{-\alpha}}\alpha$.
Applying the comparison principle to $v(Rx)$ and $(1-\e(1-R))v(x)$ the result follows. 
\end{proof}

\medskip 

With the help of Lemma \ref{lem:doubling} we can prove the existence 
of a homogeneous solution of the form $r^\kappa\phi(\sigma)$ vanishing on the boundary of the cone
$\mathcal K_1$.  

\begin{thm}\label{thm:v-exist}
There is $\kappa>0$ and $\phi\in C^\infty(\ol S)$ satisfying 
$F(D^2(r^\kappa\phi(\sigma)))=0, \phi|_{\p S}=0, \phi|_{K_1}>0$. 
\end{thm}

\begin{proof}
First we want to compare $v$ with its scalings in order to obtain two sided bounds.
From \eqref{eq:mon-R},  
the strong maximum principle and Hopf's lemma (see \cite{Giga}), we derive that there is a $k > 0$ such that
\begin{equation}\label{eq:doubling-5}
kv(x)\le v(x/2)\le k^{-1}v(x),
\end{equation}
for all $x\in \p\K_{\frac14}$. 
This and the weak comparison principle imply that \eqref{eq:doubling-5} holds for all $x\in \K_{\frac14}$.
Consequently, we can use the $C^0$-estimate of \cite{CC-Fully} and \eqref{eq:mon-R} 
(combined with the barrier argument in \cite{Wang}) in order to obtain a $v^* \in  C^0(\ol\Omega) \cap  C^1(\Omega\-\{0\})$ and a sequence of $R_k > 0$ tending to zero such that
\begin{equation}\label{eq:limit-5}
\frac{v(R_kx)}{\sup\limits_{B_{R_k}} v}\to v^*
\end{equation}
in the sense $C^0(\ol\Omega) \cap  C^1(\Omega\-\{0\})$. Moreover, 
$v^*\ge 0$ in $\K_1$, $v^*=0$ on $\p\K_1\-S$, $\|v^*\|_{L^\infty(\K_1)}=1$
and $v^*$ is a viscosity solution of  $F(D^2 v^\ast)=0$ in $\K_1$. By
the strong maximum principle (see Lemma  \ref{lem:SCP}) and Hopf's lemma \cite{Giga}, $v^* > 0$, in $\K_1$, and $\na v^*\not =0$ on $\p\K_1\-\{\{0\}\cup\ol S\}$.

\smallskip 

Let  $R\in  (0,1)$. By Lemma \ref{lem:hom-Tolk}
it is
enough to prove the existence of a $C_R > 0$ satisfying
\begin{equation}\label{eq:hom-lim}
v^*(Rx)=C_Rv^*(x).
\end{equation}

\medskip 

In order to prove \eqref{eq:hom-lim} let us define  $E_{r, R}=\{c : cu(x)\le u(Rx), \forall x\in \K_r\}$. 
By \eqref{eq:doubling-5} $E_{r, R}\not =\emptyset$. Note that 
$E_{r_2, R}\subset E_{r_1, R}$,  if $r_1<r_2$ and thus 
$$
\sup_{c\in E_{r_2, R}} c\le \sup_{c\in E_{r_1, R}} c.
$$ 
In order to show \eqref{eq:hom-lim}, we set
\begin{eqnarray*}
c_{r, R}
&=&
\sup_{E_{r, R}} c, 
\\
C_R
&=&
\sup\{c_{r, R} : r\in (0, 1)\}.
\end{eqnarray*}
By the weak comparison principle, $c_{r, R}$ is decreasing with respect  to $r\in (0, 1]$. This and \eqref{eq:limit-5}  imply that 
\begin{equation}\label{eq:hom-fail}
C_Rv^*(x)\le v^*(Rx), \quad x\in \mathcal K_1.
\end{equation}
Let us suppose that \eqref{eq:hom-lim} is not true. Then, we can use \eqref{eq:hom-fail}, the strong comparison principle and Hopf's comparison principle (see Section \ref{sec:2},  Lemmata \ref{lem:SCP} and \ref{lem:HCP}) to obtain a $\delta>0$ such that
\begin{equation*}
(C_R+2\delta)v^*(x)\le v^*(Rx), x\in \p \K_{\frac12}.
\end{equation*}
This and \eqref{eq:limit-5} show that 

\begin{equation}\label{eq:hom-fail-2}
(C_R+\delta)v(x)\le v(Rx), \forall x\in \p \K_{r}.
\end{equation}
for some $r>0$.
By the weak comparison principle, \eqref{eq:hom-fail-2} holds for all 
$x\in \K_r$. 
This, however, is a contradiction to the definition of $C_R$. Hence, \eqref{eq:hom-lim} must be true.
\end{proof}

\vs
The homogeneous solutions  $r^\kappa\phi(\sigma)$,  constructed in Theorem \ref{thm:v-exist},   provide  
two-sided control for the scalings of  the solutions of $F(D^2 u)=0$ in the cone $\K_1$.  
\begin{thm}
Let $u\ge 0$ be a viscosity solution of $F(D^2 u)=0$ in the cone $\K_1$.
Then then for every  sequence  $\{R_k\}_{k=1}^\infty$, $R_k\downarrow 0$
there exists a subsequence $R_{k_j}$ such that the functions 
\[
u_j(x)=\frac{u(x_0+R_{k_j}x)}{M(R_{k_j}, x_0)}, \quad M(R_{k_j}, x_0)=\sup_{B_{R_k}(x_0)} u
\]
converges locally uniformly to $r^\kappa \phi(\sigma)$. 
\end{thm}
\begin{proof}
Let $b=r^\kappa\phi(\sigma)$ then by the boundary Harnack principle \eqref{eq:bHarnack} there is a constant $C>0$
such that 
\begin{equation}\label{eq:nd9}
\frac1C b \le u\le C b\quad \text{ in}\  \K_{\frac12}
\end{equation}
We want to prove that 
there is a subsequence of $\{R_k\}$ such that 
\begin{equation}\label{eq:nd10}
u_k\to b
\end{equation}
in $C^0(\ol{\K_1})\cap C^1(\ol{\K_1}\-\{0\})$.

Define
\[
c_R=\sup_{E_R} c, \quad E_R= \{c\ge 0 : cb\le u, \text{in} \ \K_R\}.
\]
By \eqref{eq:nd9} $E_R\not =\emptyset$. 
From the weak comparison principle (as in the proof of Theorem \ref{thm:v-exist}), one derives that
$c_R$ 
is decreasing with respect to $R$. Consequently, the limit
\[
c^*=\lim_{R\to 0} c_R
\]
exists and it is positive. 
The $C^\alpha$-estimates of \cite{CC-Fully}, $C^1(\ol{\K_1}\-\{0\})$ regularity result 
and \eqref{eq:nd9} imply that there is a subsequence of $\{R_k\}$
and a $u^*\in C^0(\ol{\K_1})\cap C^1(\ol{\K_1}\-\{0\})$ such that 
\begin{equation}\label{eq:nd11}
u_k\to u^*
\end{equation}
in the sense of $C^0(\ol{\K_1})\cap C^1(\ol{\K_1}\-\{0\})$.
Moreover, $u^*=0$ on $\K_1\-S$, $u^*$ solves $F(D^2u^*)=0$
in $\K_1$ and 
\begin{equation}\label{eq:nd12}
c^*b\le u^*, \quad \text{in}\  \K_1. 
\end{equation}


Now, suppose that $u^*$ is not identical to $c^*b$ , in $\K_1$. 
Then, \eqref{eq:nd12}, the strong comparison principle and Hopf's comparison principle 
(see Section \ref{sec:2},  Lemmata \ref{lem:SCP} and \ref{lem:HCP})
 imply that there is a $\delta>0 $ such that
\begin{equation*}
(c^*+2\delta)b(x)\le u^*(x), x\in \p \K_{\frac12}.
\end{equation*}
This and \eqref{eq:nd11} show that 

\begin{equation}
(c^*+\delta)b(x)\le u(x), \forall x\in \p \K_{r}.
\end{equation}
for some $r>0$.
The weak comparison principle shows that this is true, also in 
$\K_r$. This, however, 
is a contradiction to the definition of $c^*$. Hence, $c^*b = u^*$, 
in $\K_1$, and \eqref{eq:nd10} is true.
\end{proof}


\end{document}